\documentclass[11pt]{amsart}
\usepackage{geometry}
\usepackage{graphicx}
\usepackage{amssymb}
\usepackage{amsmath}
\usepackage{amsthm}
\usepackage{color}
\newtheorem{definition}{Definition}
\newtheorem{lemma}{Lemma}
\newtheorem{theorem}{Theorem}
\newtheorem{assumption}{Assumption}

\newtheorem{remark}{Remark}
\newtheorem{notation}{Notation}
\newcommand{\R}{\mathbb R}
\numberwithin{equation}{section}

\title[Quasi-local energy]{Minimizing properties of critical points \\ of quasi-local energy}
\author{Po-Ning Chen, Mu-Tao Wang, and Shing-Tung Yau}
\begin{document}

\begin{abstract}

In relativity, the energy of a moving particle depends on the observer, and the rest mass is the minimal energy seen among all observers.
The Wang--Yau quasi-local mass for a surface in spacetime introduced in \cite{wy1} and \cite{wy2} is defined by minimizing quasi-local energy associated with admissible isometric embeddings of the surface into the Minkowski space. A critical point of the quasi-local energy is an isometric embedding satisfying the Euler-Lagrange equation. 
In this article, we prove results regarding both local and global minimizing properties of  critical points of the Wang--Yau quasi-local energy.
In particular, under a condition on the mean curvature vector we show a critical point minimizes the quasi-local energy locally. The same condition also implies that the critical point is globally minimizing among all axially symmetric embedding provided the image of the associated isometric embedding lies in a totally geodesic Euclidean 3-space.
\end{abstract}
\thanks{Part of this work was carried out while all three authors were visiting Department of Mathematics of National Taiwan University and Taida Institute for Mathematical Sciences in Taipei, Taiwan. M.-T. Wang is supported by NSF grant DMS-1105483 and S.-T. Yau is supported by NSF
grants DMS-0804454 and PHY-07146468.} 
\date{June  13, 2013}
\maketitle

\section{Introduction}

Let $\Sigma$ be a closed embedded spacelike 2-surface in a spacetime $N$. We assume $\Sigma$ is a topological 2-sphere and the mean curvature vector field $H$ of $\Sigma$ in $N$ is a spacelike vector field. The mean curvature vector field defines a connection one-form of the normal bundle $\alpha_{H} =\langle \nabla^N_{(\cdot)} e_3, e_4\rangle$, where $e_3=-\frac{H}{|H|}$ and $e_4$ is the future unit timelike normal vector that is orthogonal to $e_3$. Let $\sigma$ be the induced metric on $\Sigma$. 

We shall consider isometric embeddings of $\sigma$ into the Minkowski space $\R^{3,1}$. We recall that this means an embedding $X:\Sigma \rightarrow \R^{3,1}$ such that the induced metric on the image is $\sigma$.  Throughout this paper, we shall fix a constant unit timelike vector $T_0$ in $\R^{3,1}$. The time function $\tau$ on the image of the isometric embedding $X$ is defined to be $\tau = - X \cdot T_0$. The existence of such an isometric embedding is guaranteed  by a convexity condition on $\sigma$ and $\tau$ (Theorem 3.1 in \cite{wy2}). The condition is equivalent to that the metric $\sigma+ d \tau \otimes d \tau$ has positive Gaussian curvature. Let  $\widehat{\Sigma}$ be the projection of the image of $X$, $X(\Sigma)$, onto the orthogonal complement of $T_0$, a totally geodesic Euclidean 3-space in $\R^{3,1}$.  $\widehat{\Sigma}$ is a convex 2-surface in the Euclidean 3-space. Then the isometric embedding of $\Sigma$ into $\R^{3,1}$ with time function $\tau$ exists
  and is unique up to an isometry of the orthogonal complement of $T_0$.

In \cite{wy1} and \cite{wy2}, Wang and Yau define a quasi-local energy for a surface $\Sigma$ with spacelike mean curvature vector $H$ in a spacetime $N$, with respect to
an isometric embedding $X$ of $\Sigma$ into $\R^{3,1}$.  The definition relies on the physical data on $\Sigma$ which consist of the induced metric $\sigma$, the norm of the mean curvature vector $|H|>0$, and the connection one-form $\alpha_{H} $. The definition also relies on the reference data from the isometric embedding $X:\Sigma\rightarrow \R^{3,1}$ whose induced metric is the same as $\sigma$. In terms of $\tau = - X \cdot T_0$, the quasi-local energy is defined to be \footnote{This notation $E(\Sigma, \tau)$ is slightly different from \cite{wy1} and \cite{wy2} where $E(\Sigma, X, T_0)$ is considered. However, no information is lost as long as only energy and mass are considered (as opposed to energy-momentum vector).}
\begin{equation}\label{energy}\begin{split}&E(\Sigma, \tau)=\int_{\widehat{\Sigma}} \widehat{H} dv_{\widehat{\Sigma}}-\int_\Sigma \left[\sqrt{1+|\nabla\tau|^2}\cosh\theta|{H}|-\nabla\tau\cdot \nabla \theta -\alpha_H ( \nabla \tau) \right]dv_\Sigma,\end{split}\end{equation}
 where $\nabla$ and $\Delta$ are the covariant derivative and Laplace operator with respect to $\sigma$, and $\theta$ is defined by $\sinh \theta =\frac{-\Delta \tau}{|{H}|\sqrt{1+|\nabla \tau|^2}}$. Finally,  $\widehat{H}$ is the mean curvature  of  $\widehat{\Sigma}$ in $\R^3$. We note that in $E(\Sigma, \tau)$, the first argument $\Sigma$ represents a physical surface in spacetime with the data $(\sigma, |H|, \alpha_H)$, while the second argument $\tau$ indicates an isometric embedding of the induced metric into $\R^{3,1}$ with time function $\tau$ with respect to the fixed $T_0$. 
 
The Wang--Yau quasi-local mass for the surface $\Sigma$ in $N$  is defined to be the minimum of $E(\Sigma,\tau)$ among all ``admissible" time functions $\tau$ (or isometric embeddings). This admissible condition is given in Definition 5.1 of \cite{wy2}  (see also section 3.1). This condition for $\tau$ implies that $E(\Sigma, \tau)$ is non-negative if $N$ satisfies the dominant energy condition. We recall the Euler--Lagrange equation for the functional $E(\Sigma,\tau)$  of $\tau$, which is derived in \cite{wy2}. Of course, a critical point $\tau$ of $E(\Sigma, \tau)$ satisfies this equation. 
\begin{definition} Given the physical data $(\sigma, |H|, \alpha_H)$ on a 2-surface $\Sigma$. We say that a smooth function $\tau$ is a solution to the optimal 
embedding equation for $(\sigma, |H|, \alpha_H)$ if the metric $\widehat{\sigma}=\sigma+d\tau\otimes d\tau$ can be isometrically embedded into $\R^3$ with image $\widehat\Sigma$ such that
\begin{equation}\label{Euler}
-(\widehat{H}\hat{\sigma}^{ab} -\hat{\sigma}^{ac} \hat{\sigma}^{bd} \hat{h}_{cd})\frac{\nabla_b\nabla_a \tau}{\sqrt{1+|\nabla\tau|^2}}+ div_\sigma (\frac{\nabla\tau}{\sqrt{1+|\nabla\tau|^2}} \cosh\theta|{H}|-\nabla\theta-\alpha_H)=0\end{equation}
where $\nabla$ and $\Delta$ are the covariant derivative and Laplace operator with respect to $\sigma$, $\theta$ is defined by $\sinh \theta =\frac{-\Delta \tau}{|{H}|\sqrt{1+|\nabla \tau|^2}}$.  $\widehat{h}_{ab}$ and $\widehat{H}$ are the second fundamental form and the mean curvature of $\widehat{\Sigma}$, respectively.
\end{definition}
It is natural to ask the following questions:
\begin{enumerate}
\item How do we find solutions to the optimal embedding equation?
\item Does a solution of the optimal embedding equation minimize $E(\Sigma, \tau)$, either locally or globally?
\end{enumerate}

Before addressing these questions, let us fix the notation on the space of isometric embeddings of $(\Sigma, \sigma)$ into $\R^{3,1}$.
 \begin{notation}
$T_0$ is  a fixed constant future timelike unit vector throughout the paper, $X_\tau$ will denote an isometric embedding of $\sigma$ into $\R^{3,1}$ with time function $\tau=-X\cdot T_0$.  We denote the image of $X_\tau$ by $\Sigma_\tau$, the mean curvature vector of $\Sigma_\tau$ by $H_\tau$ and the connection one-form of the normal bundle of $\Sigma_{\tau}$ determined by $H_\tau$ by $\alpha_{H_{\tau}}$. We denote the projection of $\Sigma_\tau$ onto the orthogonal complement of $T_0$ by $\widehat{\Sigma}_\tau$  and the mean curvature of  $\widehat{\Sigma}_\tau$ by  $\widehat H_{\tau}$.
\end{notation}

In particular, when $\sigma$ has positive Gauss curvature, $X_0$ will denote an isometric embedding of $\sigma$ into the orthogonal complement of $T_0$. $\Sigma_0$ denotes the image of $X_0$, The mean curvature $H_0$ of $\Sigma_0$ can be viewed as a positive function by the assumption.

 In \cite{cwy}, the authors  studied the above questions at spatial or null infinity of asymptotically flat manifolds and proved that a series solution exists for the optimal embedding equation, the solution minimizes the quasi-local energy locally and the mass it achieves agrees with the ADM or Bondi mass at infinity for asymptotically flat manifolds. On the other hand, when $\alpha_H$ is divergence free, $\tau=0$ is a solution to the optimal embedding equation.   Miao--Tam--Xie \cite{mtx} and Miao--Tam \cite{mt} studied the time-symmetric case and found several conditions such that $\tau=0$ is a local minimum. 
In particular, this holds if
$ H_0 > |H|>0 $
where $H_0$ is the mean curvature of the isometric embedding $X_0$ of $\Sigma$ into $\R^3$ (i.e. with time function $\tau=0$).  

For theorems proved in this paper, we impose the following assumption on the physical surface $\Sigma$.

\begin{assumption}\label{assumption}
Let $\Sigma$ be a closed embedded spacelike 2-surface in a spacetime $N$ satisfying the dominant energy condition. We assume that $\Sigma$ is a topological 2-sphere and the mean curvature vector field $H$ of $\Sigma$ in $N$ is a spacelike vector field.
\end{assumption}

We first prove the following comparison theorem among quasi-local energies. 
\begin{theorem} \label{thm_int}
Suppose  $\Sigma$ satisfies Assumption 1 and $\tau_0$ is a is a critical point of the quasi-local energy functional $E(\Sigma,\tau)$.  
Assume further that 
\[ |H_{\tau_0}| > |H| \]
where $H_{\tau_0}$ is the mean curvature vector of the isometric embedding of $\Sigma$ into $\R^{3,1}$ with time function $\tau_0$.
Then,  for any time function $\tau$ such that $ \sigma + d \tau \otimes d \tau$ has positive Gaussian curvature, we have
\[E(\Sigma,\tau) \ge E(\Sigma,\tau_0) +E(\Sigma_{\tau_0}, \tau).\]
Moreover, equality holds if and only if $\tau-\tau_0$ is a constant .  
\end{theorem}

As a corollary of Theorem \ref{thm_int}, we prove the following theorem about local minimizing property of an arbitrary, non-time-symmetric, solution to the optimal embedding equation.
\begin{theorem} \label{thm_local}
Suppose  $\Sigma$ satisfies Assumption 1 and $\tau_0$ is a critical point of the quasi-local energy functional $E(\Sigma, \tau)$.
Assume further that 
\[ |H_{\tau_0}| > |H| >0\]
where $H_{\tau_0}$ is the mean curvature vector of the isometric embedding $X_{\tau_0}$ of $\Sigma$ into $\R^{3,1}$ with time function $\tau_0$.
Then, 
 $\tau_0$ is a local minimum for  $E(\Sigma,\tau)$.
\end{theorem}
The special case when $\tau_0=0$ was proved by Miao--Tam--Xie  \cite{mtx}. They estimate the second variation of quasi-local energy around 
the critical point $\tau=0$  by linearizing the optimal embedding equation near the critical point and then applying a generalization of Reilly's formula. As we allow
$\tau_0$ to be an arbitrary solution of the optimal isometric embedding equation, a different method is devised to deal with the fully nonlinear nature of the equation.

In general, the space of admissible isometric embeddings as solutions of a fully nonlinear elliptic system is very complicated and global knowledge of the quasi-local energy is difficult to obtained. However, we are able to prove a global minimizing result in the axially symmetric case. 

\begin{theorem} \label{thm_global_axial}
Let  $\Sigma$ satisfy Assumption 1. Suppose that the induced metric $\sigma$ of  $\Sigma$  is axially symmetric with positive Gauss curvature, $\tau=0$ is a solution to the optimal embedding equation for $\Sigma$ in $N$, and
\[ H_0 > |H|>0. \]
Then for any axially symmetric time function $\tau$ such that  $\sigma + d\tau \otimes  d\tau$ has positive Gauss curvature, 
\[ E(\Sigma , \tau) \ge E(\Sigma, 0). \]
Moreover, equality holds if and only if $\tau$ is a constant . 
\end{theorem}

This theorem will have applications in studying quasi-local energy in the Kerr spacetime, or more generally an axially symmetric spacetime. It is very likely that the global minimum
of quasilocal energy of an axially symmetric datum is achieved at an axially symmetric isometric embedding into the Minkowski, though we cannot prove it at this moment.  In 
\cite{mtx}, Miao, Tam and Xie described several situations where the condition $H_0 > |H|$ holds. In particular, this includes large spheres in Kerr spacetime.



In section 2, we prove Theorem \ref{thm_int} using the nonlinear structure of the quasi-local energy.
In section 3, we  prove the admissibility of the time function $\tau$ for the surface $\Sigma_{\tau_0}$  in $\R^{3,1}$ when  $\tau$ is close to $\tau_0$. The positivity of quasi-local mass follows from the admissibility of the time function. Combining with Theorem \ref{thm_int},  this proves Theorem \ref{thm_local}.  In section 4, we prove Theorem  \ref{thm_global_axial}. Instead of using admissibility, we prove the necessary positivity of quasi-local energy using variation of quasi-local energy and a point-wise mean curvature inequality.

\section{A comparison theorem for quasi-local energy}
In this section, we prove Theorem  \ref{thm_int}. 


\begin{proof}
We start with a metric $\sigma$ and consider an isometric embedding into $\R^{3,1}$ with time function $\tau_0$. The image is an embedded space-like 2-surface $\Sigma_{\tau_0}$ in $\R^{3,1}$.
The corresponding data on $\Sigma_{\tau_0}$ are denoted as $|H_{\tau_0}|$ and $ \alpha_{H_{\tau_0}}$.
We consider the quasi-local energy  of $\Sigma_{\tau_0}$ as a physical surface in the spacetime $\R^{3,1}$ with respect to another isometric embedding $X_\tau$ into $\R^{3,1}$ with time function $\tau$. 
We recall that
\[\begin{split}
 E(\Sigma_{\tau_0},\tau) =\int_{\widehat{\Sigma}_{\tau}} \widehat{H} dv_{\widehat{\Sigma}_{\tau}}-\int_\Sigma \left[\sqrt{1+|\nabla\tau|^2}\cosh\theta_{(\tau,\tau_0)}|{H_{\tau_0}}|-\nabla\tau\cdot \nabla \theta_{(\tau,\tau_0)} - \alpha_{H_{\tau_0}}( \nabla \tau) \right]dv_\Sigma \end{split}\] where $\theta_{(\tau,\tau_0)}$ is defined by $\sinh \theta_{(\tau,\tau_0)} =\frac{-\Delta \tau}{|{H_{\tau_0}}|\sqrt{1+|\nabla \tau|^2}}$.

Using $ E(\Sigma_{\tau_0},\tau)$, $E(\Sigma,\tau)$ can be expressed as 
\begin{equation}\label{firstlowerbound}
\begin{split}
E(\Sigma,\tau) & =   \int_{\widehat{\Sigma}} \widehat{H} dv_{\widehat{\Sigma}}  -\int_\Sigma \left[\sqrt{1+|\nabla\tau|^2}\cosh\theta|{H}|-\nabla\tau\cdot \nabla \theta -\alpha_H ( \nabla \tau) \right]dv_\Sigma\\
& =  \,E(\Sigma_{\tau_0},\tau) + A
\end{split} 
\end{equation} where
\begin{equation}
\begin{split}
A&= \int_\Sigma \left[\sqrt{1+|\nabla\tau|^2}\cosh\theta_{(\tau,\tau_0)}|{H_{\tau_0}}|-\nabla\tau\cdot \nabla \theta_{(\tau,\tau_0)} - \alpha_{H_{\tau_0}}( \nabla \tau) \right]dv_\Sigma \\
& -\int_\Sigma \left[\sqrt{1+|\nabla\tau|^2}\cosh\theta|{H}|-\nabla\tau\cdot \nabla \theta - \alpha_H ( \nabla \tau ) \right]dv_\Sigma.
\end{split}
\end{equation}
In the following, we shall show that $A\geq E(\Sigma,\tau_0)$.

One can rewrite $ div_{\sigma}  \alpha_H $ and $div_{\sigma}  \alpha_{H_{\tau_0}} $ using the optimal embedding equation.
First, $\tau_0$ is a solution to the original optimal embedding equation. We have
\begin{equation}  \label{div_1}
\begin{split}
 div_{\sigma}  \alpha_H =& -(\widehat{H}\hat{\sigma}^{ab} -\hat{\sigma}^{ac} \hat{\sigma}^{bd} \hat{h}_{cd})\frac{\nabla_b\nabla_a \tau_0}{\sqrt{1+|\nabla\tau_0|^2}}+ div_\sigma \left[ \frac{\nabla\tau_0}{\sqrt{1+|\nabla\tau_0|^2}}  \sqrt{ H^2+ \frac{(\Delta \tau_0)^2}{1+|\nabla \tau_0|^2} }\right] \\
&+ \Delta  \left[  \sinh^{-1} \frac{\Delta \tau_0}{|{H}|\sqrt{1+|\nabla \tau_0|^2}} \right]
\end{split}
\end{equation} 
where $\hat{h}_{ab}$ and $\widehat{H}$ are the second fundamental form and mean curvature of $\widehat \Sigma_{\tau_0}$, respectively.

On the other hand, $\tau=\tau_0$ locally minimizes $E( \Sigma_{\tau_0},\tau)$ by the positivity of quasi-local energy. Hence,
\begin{equation}  \label{div_2}
\begin{split}
div_{\sigma}  \alpha_{H_{\tau_0}} =& -( \widehat{H}\hat{\sigma}^{ab} -\hat{\sigma}^{ac} \hat{\sigma}^{bd} \hat{h}_{cd})\frac{\nabla_b\nabla_a \tau_0}{\sqrt{1+|\nabla\tau_0|^2}}+ div_\sigma \left[ \frac{\nabla\tau_0}{\sqrt{1+|\nabla\tau_0|^2}}  \sqrt{ H_{\tau_0}^2+ \frac{(\Delta \tau_0)^2}{1+|\nabla \tau_0|^2} }\right] \\
&+ \Delta  \left[ \sinh^{-1} \frac{\Delta \tau_0}{|{H_{\tau_0}}|\sqrt{1+|\nabla \tau_0|^2}}\right]
\end{split}
\end{equation} 
where $\hat{h}_{ab}$ and $\widehat{H}$ are  the same as in equation \eqref{div_1}  (this can be verified directly as in \cite{w}).

Using equations (\ref{div_1}) and (\ref{div_2}), we have
\[\begin{split}
A = \int_{\Sigma}&  \sqrt{(1+|\nabla\tau|^2)|{H_{\tau_0}}|^2 + (\Delta \tau)^2 } 
 -  \sqrt{(1+|\nabla\tau|^2)|H|^2 + (\Delta \tau)^2 } \\
& -\Delta \tau \sinh^{-1}(\frac{\Delta \tau}{ |{H_{\tau_0}}| \sqrt{1+|\nabla\tau|^2}})+\Delta \tau \sinh^{-1}(\frac{\Delta \tau}{ |H| \sqrt{1+|\nabla\tau|^2}}) \\
& +\Delta \tau \sinh^{-1}(\frac{\Delta \tau_0}{ |{H_{\tau_0}}| \sqrt{1+|\nabla\tau_0|^2}})-\Delta \tau \sinh^{-1}(\frac{\Delta \tau_0}{ |H| \sqrt{1+|\nabla\tau_0|^2}}) \\
&-\frac{\nabla \tau_0 \cdot \nabla \tau}{\sqrt{1+|\nabla\tau_0|^2}}  \Big{[} \sqrt{ |H_{\tau_0}|^2+ \frac{(\Delta \tau_0)^2}{1+|\nabla \tau_0|^2} }- \sqrt{ |H|^2+ \frac{(\Delta \tau_0)^2}{1+|\nabla \tau_0|^2} } \Big{]}.
\end{split}\]
With a new  variable $x= \frac{\Delta \tau}{\sqrt{1+|\nabla \tau|^2}}$ and  $x_0= \frac{\Delta \tau_0}{\sqrt{1+|\nabla \tau_0|^2}}$,
the first six terms of the integrand is simply
\[
 \begin{split} 
 \sqrt{1+|\nabla \tau|^2} {\Big [}& \sqrt{ |H_{\tau_0}|^2+  x^2 }- \sqrt{ |H|^2+x^2}  \\
 & -x \big ( \sinh^{-1}{\frac{x}{ |H_{\tau_0}|}}- \sinh^{-1}\frac{x}{ |H|} -\sinh^{-1}\frac{x_0}{ |H_{\tau_0}|}+\sinh^{-1}\frac{x_0}{ |H|}\big){\Big ]}.
\end{split}
  \]
Let
\[  \begin{split} 
f(x) = &  \sqrt{ |H_{\tau_0}|^2+  x^2 }- \sqrt{ |H|^2+x^2}  \\
 & -x \left[ \sinh^{-1}{\frac{x}{ |H_{\tau_0}|}}- \sinh^{-1}\frac{x}{ |H|} -\sinh^{-1}\frac{x_0}{ |H_{\tau_0}|}+\sinh^{-1}\frac{x_0}{ |H|}\right].
\end{split} \]
Direct computation shows that
\[  f'(x) =\sinh^{-1}\frac{x}{ |H|} -\sinh^{-1}{\frac{x}{ |H_{\tau_0}|}} +\sinh^{-1}\frac{x_0}{ |H_{\tau_0}|}-\sinh^{-1}\frac{x_0}{ |H|}. \]
If $ |H_{\tau_0}| > |H|$, $ f(x_0) =   \sqrt{ |H_{\tau_0}|^2+  x_0^2 }- \sqrt{ |H|^2+x_0^2}$ is the global minimum for $f$(x) and equality holds if and only if $x=x_0$.
 Hence,
\[ 
\begin{split} 
A  &\ge \int_{\Sigma} ( \sqrt{1+|\nabla \tau|^2} - \frac{\nabla \tau_0 \cdot \nabla \tau}{\sqrt{1+|\nabla\tau_0|^2}}) \Big{ [} \sqrt{ |H_{\tau_0}|^2+ \frac{(\Delta \tau_0)^2}{1+|\nabla \tau_0|^2} }- \sqrt{ |H|^2+ \frac{(\Delta \tau_0)^2}{1+|\nabla \tau_0|^2} } \Big{]}    \\
 &\ge \int_{\Sigma} (  \frac{1}{\sqrt{1+|\nabla\tau_0|^2}})\Big{ [} \sqrt{ |H_{\tau_0}|^2+ \frac{(\Delta \tau_0)^2}{1+|\nabla \tau_0|^2} }- \sqrt{ |H|^2+ \frac{(\Delta \tau_0)^2}{1+|\nabla \tau_0|^2} } \Big{]}  
\end{split} \]
The last inequality follows simply from
\[  ( \sqrt{1+|\nabla \tau|^2} - \frac{\nabla \tau_0 \cdot \nabla \tau}{\sqrt{1+|\nabla\tau_0|^2}}) \ge     ( \sqrt{1+|\nabla \tau|^2} - \frac{ | \nabla \tau_0| |\nabla \tau|}{\sqrt{1+|\nabla\tau_0|^2}}) \ge   \frac{1}{\sqrt{1+|\nabla\tau_0|^2}}\] and equality holds if and only if $\nabla \tau=\nabla \tau_0$.
On the other hand,
\[  E(\Sigma,\tau_0) =  \int_{\Sigma} (  \frac{1}{\sqrt{1+|\nabla\tau_0|^2}})\Big{ [} \sqrt{ |H_{\tau_0}|^2+ \frac{(\Delta \tau_0)^2}{1+|\nabla \tau_0|^2} }- \sqrt{ |H|^2+ \frac{(\Delta \tau_0)^2}{1+|\nabla \tau_0|^2} } \Big{]}    \]
if one evaluates $div_{\sigma}  \alpha_{H}$ and $div_{\sigma}  \alpha_{H_{\tau_0}}$  using equations (\ref{div_1}) and (\ref{div_2}).
\end{proof}


\section{Local minimizing property of critical points of quasi-local energy}

In this section, we start with a metric $\sigma$ and consider an isometric embedding into $\R^{3,1}$ with time function $\tau_0$. The image is an embedded space-like 2-surface $\Sigma_{\tau_0}$ in $\R^{3,1}$.
The corresponding data on $\Sigma_{\tau_0}$ are denoted as $|H_{\tau_0}|$ and $ \alpha_{H_{\tau_0}}$.
We consider the quasi-local energy  of $\Sigma_{\tau_0}$ as a physical surface in the spacetime $\R^{3,1}$ with respect to another isometric embedding $X_\tau$ into $\R^{3,1}$ with time function $\tau$.  In this section, we prove that for $\tau$ close to $\tau_0$, $\tau$ is admissible with respect to the surface $\Sigma_{\tau_0}$. This shows that, 
for $\tau$ close to $\tau_0$, we have
\[  E(\Sigma_{\tau_0} ,\tau)  \ge 0 \]
\subsection{Admissible isometric embeddings}
We recall definitions and theorems related to admissible isometric embeddings from \cite{wy2}.

\begin{definition}\label{generalized_mean}
Suppose $i: \Sigma \hookrightarrow N$ is an embedded spacelike two-surface in a spacetime $N$. Given a smooth function $\tau$ on $\Sigma$ and a spacelike unit normal $e_3$,
the generalized mean curvature associated with these data is defined to be 
\[ h(\Sigma, i,\tau ,e_3) = - \sqrt{1+|\nabla \tau|^2} \langle H , e_3 \rangle - \alpha_{e_3}(\nabla \tau)  \]
where, as before, $H$ is the mean curvature vector of $\Sigma$ in $N$ and $\alpha_{e_3}$ is the connection form of the normal bundle determined by $e_3$. 
\end{definition}
Recall, in Definition 5.1 of \cite{wy2}, given a physical surface $\Sigma$ in spacetime $N$ with induced metric $\sigma$ and mean curvature vector $H$,  there are three conditions for a function $\tau$ to be admissible.
\begin{enumerate}
\item{ The metric, $\sigma +d \tau \otimes d \tau$, has positive Gaussian curvature.}

\item{ $\Sigma$ bounds a hypersurface $\Omega$ in $N$ where Jang's equation with Dirichlet boundary condition $\tau$ is solvable on $\Omega$ . Let the solution be $f$.}

\item{The generalized mean curvature $h(\Sigma, i,\tau ,e_3')$  is positive where $e_3'$ is determined by the solution $f$ of Jang's equation as follows:}
\[ e_3' = \cosh \theta e_3 + \sinh \theta e_4 \]
where $\sinh \theta = \frac{e_3(f)}{\sqrt{1+|\nabla \tau|^2}}$ and $e_3$ is the outward unit spacelike normal of $\Sigma$ in $\Omega$, $e_4$ is the future timelike unit normal of $\Omega$ in $N$.
\end{enumerate} 

We recall that from \cite{wy2} if $\tau$ corresponds to an admissible isometric embedding and $\Sigma$ is a 2-surface in spacetime $N$
that satisfies Assumption 1, then the quasi-local energy $E(\Sigma, \tau)$ is non-negative.

\subsection{Solving Jang's equation}
Let $(\Omega, g_{ij})$ be a Riemannian manifold with boundary $\partial \Omega=\Sigma$. Let $p_{ij}$ be
a symmetric 2-tensor on $\Omega$. The Jang's equation asks for a hypersurface $\widetilde \Omega$ in $\Omega \times \R$, defined as a graph of a function $f$ over $\Omega$, such that the mean curvature of 
$\widetilde \Omega$  is the same as the trace of the restriction of $p$ to $\widetilde \Omega$ . In a local coordinate $x^i$ on $\Omega$, Jang's equation takes the following form: 
\[   \sum_{i,j=1}^3(g^{ij }   -\frac{f^i f^j}{1+|Df|^2}) (\frac{D_iD_j f}{\sqrt{1+ |Df|^2}} -p_{ij}) =0  \]
where $D$ is the covariant derivative with respect to the metric $g_{ij}$. 
When $p_{ij}=0$, Jang's equation becomes the equation for minimal graph. Equation of of minimal surface type may have blow-up solutions. 
In \cite{sy}, it is shown by Schoen-Yau that solutions of Jang's equation can only blow-up at marginally trapped surface in $\Omega$. Namely, surfaces $S$ in $\Omega$ such that
\[  H_S  \pm tr_S p =0.  \] 

Following the analysis of Jang's equation in \cite{sy} and \cite{wy2},  we prove the following theorem for the existence of solution to the Dirichlet problem of Jang's equation. 

\begin{theorem}\label{thm_Jang}
Let $(\Omega, g_{ij})$ be a Riemannian manifold with boundary $\partial \Omega=\Sigma$, $p_{ij}$ be
a symmetric 2-tensor on $\Omega$, and $\tau$ be a function on $\Sigma$. Then Jang's equation with Dirichlet boundary data $\tau$ is solvable on $\Omega$ if
\[ H_\Sigma>| tr_\Sigma p|, \]
 and  there is no marginally trapped surface inside $\Omega$.
\end{theorem}
Following the approach in \cite{sy} and  Section 4.3 of \cite{wy2}, it suffices to  control  the boundary gradient of the solution to Jang's equation. 
 We have the following theorem:
\begin{theorem} \label{thm_solvable_Jang}
The normal derivative of a solution of the Dirichlet problem of Jang's equation is bounded if
\[ H_\Sigma> |tr_\Sigma p|. \]
\end{theorem}
\begin{remark}In \cite{aem}, Andersson, Eichmair and Metzger proved a similar 
result about boundedness of boundary gradient of solutions to Jang's equation in order to study existence of marginally trapped surfaces.
\end{remark}
\begin{proof}
We follow the approach used in  Theorem 4.2 of \cite{wy2} but with a different form for the sub and super
solutions. Let $g_{ij}$ and $p_{ij}$ be the induced metric and second fundamental form of the hypersurface  $\Omega$.  Let $\Sigma$ be the boundary of $\Omega$ with induced metric 
$\sigma$ and mean curvature $H_{\Sigma}$ in $\Omega$. We consider the following operator for Jang's equation.
\[  Q(f) = \sum_{i,j=1}^3 (g^{ij}  -\frac{f^i f^j}{1+|Df|^2}) (\frac{D_iD_j f}{\sqrt{1+|Df|^2} } -p_{ij}).   \] 
We extend the boundary data $\tau$ to the interior of the hypersurface $\Omega$. We still denote the extension by $\tau$. Consider the following test function
\[  f= \frac{\Psi(d)}{\epsilon} + \tau  \] 
where $d$ is the distance function to the boundary of $\Omega$. We compute
\[ 
\begin{split} 
D_i f= &\frac{1}{\epsilon} \Psi' d_i + D_i\tau \\
D_i D_j f  =& \frac{1}{\epsilon}( \Psi'' d_i d_j +\Psi'  D_i D_j d ) + D_iD_j \tau
\end{split}
 \]
Therefore, 
\[   
\begin{split}
Q(f) = &  (g^{ij}  -\frac{f^i f^j}{1+|Df|^2}) (\frac{D_iD_j f}{\sqrt{1+|Df|^2} } -p_{ij})  \\
         =  &  \frac{1}{\epsilon}  (g^{ij}  -\frac{f^i f^j}{1+|Df|^2})\frac{ \Psi'' d_i d_j }{\sqrt{1+|Df|^2}}+    \frac{1}{\epsilon}  (g^{ij}  -\frac{f^i f^j}{1+|Df|^2}) \frac { \Psi'  D_i D_j d}{\sqrt{1+|Df|^2}}  \\
             &   +  (g^{ij}  -\frac{f^i f^j}{1+|Df|^2})\frac{D_iD_j \tau }{\sqrt{1+|Df|^2}}+  (g^{ij}  -\frac{f^i f^j}{1+|Df|^2}) p_{ij}
\end{split}
\]
At the boundary of $\Omega$, it is convenient to use a frame $\{ e_a, e_3 \}$ where $e_a$ are tangent to the boundary and $e_3$ is normal to the boundary. 
\[
\begin{split}
 D_3 f = & \frac{1}{\epsilon} \Psi' + D_n\tau  \\
D_a f = & D_a \tau .
\end{split}
\]
As $\epsilon$ approaches  $0$, we have
\[
\begin{split}
    \epsilon \sqrt{1+|Df|^2}  = & |\Psi'| +  O(\epsilon)\\
   (g^{ij}  -\frac{f^i f^j}{1+|Df|^2})  =&  \sigma^{ab}+ O(\epsilon) .
\end{split}
\]
Moreover, the distance function $d$ to the boundary of $\Omega$ satisfies
\[  \sigma^{ab} d_{a} d_{b}= 0   \qquad {\rm and} \qquad  \sigma^{ab}D_aD_b d=H_{\Sigma}.\]

Hence, we conclude that 
\[  Q(f)  = \frac{\Psi'}{|\Psi'|   } H_{\Sigma}+ tr_{\Sigma} p + O(\epsilon) \] 
As a result,  sub and super solutions exist when  $H_\Sigma> |tr_\Sigma p|$. One can then use Perron method to find the solution between the sub and super solution. 
\end{proof}
\begin{remark}
Here we proved the result when the dimension of $\Omega$ is 3. The result holds in higher dimension as well. 
\end{remark}
\subsection{Proof of Theorem 2}

\begin{proof}
It suffices to show that, for $\tau$ close to $\tau_0$, $\tau$ is admissible with respect to $\Sigma_{\tau_0}$ in $\R^{3,1}$ if
$H_{{\tau_0}}$ is spacelike.

For such a $\tau$, the Gauss curvature of $\sigma + d \tau \otimes  d\tau $ is close to that of  $\sigma + d \tau_0 \otimes  d\tau_0 $. In particular, it remains positive.  This verifies the first condition for admissibility.

We apply  Theorem \ref{thm_Jang} in the case where $\Sigma$ bounds a spacelike hypersurface $\Omega$ in $\R^{3,1}$ and $g_{ij}$ and $p_{ij}$ be the induced metric and second fundamental form on $\Omega$, respectively. Moreover, the projection of $\Sigma$ onto the orthogonal complement of $T_0$ is convex since  the Gauss curvature of $\sigma + d \tau \otimes  d\tau $ is positive. 

That the mean curvature of $\Sigma$ is spacelike implies that $ |H_\Sigma| >| tr_\Sigma p|$. Moreover, that the projection $\widehat \Sigma$ is convex implies $H_\Sigma>0$. It follows that  $ H_\Sigma >| tr_\Sigma p|$.  We recall  
 there is no marginally trapped surface in $\R^{3,1}$ (see for example, \cite{k}).  As a result, Jang's equation with the Dirichlet boundary data $\tau$ is solvable as long as $\Sigma=\partial \Omega$ has spacelike mean curvature vector. This verifies the second condition for  admissibility.

To verify the last condition,  it suffices to check for $\tau_0$ since it is an open condition. Namely, it suffices to prove that 
\[ h(\Sigma, X_{\tau_0}, \tau_0 ,  e_3') >0. \] 
Lemma \ref{lemma_gauge} in the appendix implies that the generalized mean curvature $h(\Sigma, X_{\tau_0}, \tau_0 ,e'_3)$ is the same as the generalized mean curvature $h(\Sigma, X_{\tau_0}, \tau_0 , \breve e_3(\Sigma_{\tau_0}))$  where $\breve e_3(\Sigma_{\tau_0})$ is the vector field on $\Sigma_{\tau_0}$ obtained by parallel translation of the outward unit normal of  $\widehat \Sigma_{\tau_0}$ along $T_0$.  The last condition for $\tau_0$ now follows from Proposition 3.1 of \cite{wy2}, which states that for $\breve e_3 (\Sigma_{\tau_0})$,
\[ h(\Sigma, X_{\tau_0}, \tau_0 , \breve e_3 ( \Sigma_{\tau_0})) =  \frac{\widehat H}{\sqrt{1+ |\nabla \tau_0|^2}}>0. \] 
\end{proof}

\section{Global minimum in the axially symmetric case}
 In proving Theorem 3, we need three Lemmas concerning a spacelike 2-surface $\Sigma_\tau$ in $\R^{3,1}$ with time function $\tau$. 
First we introduce a new energy functional which depends on a physical gauge. 
\begin{definition} Let $\Sigma$ be closed embedded spacelike 2-surface in spacetime $N$ with induced metric $\sigma$. Let $e_3$ be a spacelike normal vector field along $\Sigma$ in $N$. For any $f$ such that the isometric  embedding of $\sigma$ into $\R^{3,1}$ with time function $f$ exists, we define
\begin{equation}   \label{new_energy}
\tilde E(\Sigma, e_3,f) =   \int_{\widehat \Sigma_{f}}  \widehat H_{f} dv_{\widehat \Sigma_{f}} - \int_{ \Sigma}  h(\Sigma, i, f, e_3)   dv_{\Sigma}  
\end{equation} where $h(\Sigma, i, f, e_3)$ is the generalized mean curvature (see Definition \ref{generalized_mean}).
\end{definition}

This functional is less nonlinear than $E(\Sigma, f)$. Provided  the mean curvature vector $H$ of $\Sigma$ is spacelike, the following relation holds
\[E(\Sigma, f)=\tilde{E}(\Sigma, {e}_3^{\text{can}}(f), f)\]
where ${e}_3^{\text{can}}(f)$ is chosen such that 
\[     \langle H, {e}_4^{\text{can}}(f) \rangle= \frac{- \Delta f}{|H| \sqrt{1+|\nabla f|^2}}    \]

In addition, the first variation of  $\tilde E(\Sigma, e_3, f)$ with respect to $f$ can be computed as in \cite{wy2}:
\[
 (\widehat{H}\hat{\sigma}^{ab} -\hat{\sigma}^{ac} \hat{\sigma}^{bd} \hat{h}_{cd})\frac{\nabla_b\nabla_a f}{\sqrt{1+|\nabla f|^2}}+ div_\sigma (  \frac{\langle H , e_3 \rangle \nabla f}{\sqrt{1+|\nabla f|^2}}) +  div_{\sigma} \alpha_{ e_3}.
\]
where $\hat{h}_{ab}$ and $\widehat{H}$ are the second fundamental form and mean curvature of $\widehat \Sigma_{f}$, respectively.

We recall that for $\Sigma_\tau$, assuming the projection onto the orthonormal complement of $T_0$ is an embedded surface, there is a unique outward normal spacelike unit vector field $\breve{e}_3 (\Sigma_\tau)$ which is orthogonal to $T_0$. Indeed, $\breve{e}_3 (\Sigma_\tau)$ can be obtained by parallel translating the unit outward normal vector of $\widehat \Sigma_{\tau}$, $\hat \nu$, along $T_0$.

\begin{lemma} \label{lemma_cri}
For a spacelike 2-surface $\Sigma_\tau$ in $\R^{3,1}$ with time function $\tau$,  $f=\tau$ is a critical point of the functional $\tilde E(\Sigma_{\tau}, \breve{e}_3(\Sigma_\tau),  f)$.
\end{lemma}
\begin{proof}
The first variation of  $\tilde E(\Sigma_{\tau}, \breve{e}_3(\Sigma_\tau), f)$ at $f = \tau$ is 
\[
 (\widehat{H}\hat{\sigma}^{ab} -\hat{\sigma}^{ac} \hat{\sigma}^{bd} \hat{h}_{cd})\frac{\nabla_b\nabla_a \tau}{\sqrt{1+|\nabla\tau|^2}}+ div_\sigma (  \frac{h \nabla \tau}{\sqrt{1+|\nabla\tau|^2}}) +  div_{\sigma} \alpha.
\]
where $h=\langle H_\tau, \breve{e}_3(\Sigma_\tau)\rangle $ and $\alpha=\alpha_{\breve{e}_3(\Sigma_\tau)}$ are data on $\Sigma_\tau$ with respect to the gauge $\breve{e}_3(\Sigma_\tau)$ and 
$ \hat{h}_{ab}$  and $\widehat{H} $  are the second fundamental form and mean curvature of $\widehat \Sigma_{\tau}$. Denote the covariant derivative on $\widehat \Sigma$  with respect to the induced metric $\hat \sigma$ by $\hat \nabla$.
We wish to show that the first variation is $0$.

Recall from \cite{wy2}, we have
\[  
 \widehat H  = - h -\frac{ \alpha ( \nabla \tau) }{1+|\nabla \tau|^2}.
\]

Moreover, we have the following relation between metric and covariant derivative of $\Sigma_{\tau}$ and $\widehat \Sigma_{\tau}$:
\[
\begin{split}
\hat \sigma ^{ab}  = \sigma^{ab} - \frac{\tau^a \tau^b}{ 1+|\nabla \tau|^2}, \,\,
\hat \sigma ^{ab}  \hat \nabla_a \tau  = \frac{\tau^b}{ 1+|\nabla \tau|^2},
\text{ and } \hat \nabla_a  \hat \nabla_b \tau & = \frac{\nabla_a \nabla_b \tau }{ 1+|\nabla \tau|^2}.
\end{split}
\]

In addition, for a tangent vector field $W^a$,
$ \hat \nabla_a W^a = \nabla_a W^a + \frac{ (\nabla_b  \nabla_c \tau) \tau^c W^b}{ 1+|\nabla \tau|^2}.$
As a result, 
\[
\begin{split}
   & (\widehat{H}\hat{\sigma}^{ab} -\hat{\sigma}^{ac} \hat{\sigma}^{bd} \hat{h}_{cd})\frac{\nabla_b\nabla_a \tau}{\sqrt{1+|\nabla\tau|^2}} \\
={}  & \hat \nabla_b  [ (\widehat{H}\hat{\sigma}^{ab} -\hat{\sigma}^{ac} \hat{\sigma}^{bd} \hat{h}_{cd})    \sqrt{1+|\nabla\tau|^2} \hat \nabla_a \tau ] -  (\widehat{H}\hat{\sigma}^{ab} -\hat{\sigma}^{ac} \hat{\sigma}^{bd} \hat{h}_{cd})  (\hat \nabla_a \tau )(\hat \nabla_b \sqrt{1+ |\nabla \tau|^2})\\
={} &  \nabla_b  [ (\widehat{H}\hat{\sigma}^{ab} -\hat{\sigma}^{ac} \hat{\sigma}^{bd} \hat{h}_{cd})   \sqrt{1+|\nabla\tau|^2} \hat \nabla_a   \tau] + \frac{ \tau^e ( \nabla_b \nabla _e \tau)}{1+|\nabla \tau|^2} (\widehat{H}\hat{\sigma}^{ab} -\hat{\sigma}^{ac} \hat{\sigma}^{bd} \hat{h}_{cd})  \sqrt{1+|\nabla\tau|^2} \hat \nabla_a\tau \\
  {} &    -   (\widehat{H}\hat{\sigma}^{ab} -\hat{\sigma}^{ac} \hat{\sigma}^{bd} \hat{h}_{cd})  \hat \nabla_a \tau \frac{  \tau^e \nabla_b \nabla _e \tau}{\sqrt{1+|\nabla \tau|^2}} \\
={} &  \nabla_b  [ (\widehat{H}\hat{\sigma}^{ab} -\hat{\sigma}^{ac} \hat{\sigma}^{bd} \hat{h}_{cd})   \sqrt{1+|\nabla\tau|^2} \hat \nabla_a  \tau]. 
\end{split}
 \]
Hence, 
\[
\begin{split}
 &   (\widehat{H}\hat{\sigma}^{ab} -\hat{\sigma}^{ac} \hat{\sigma}^{bd} \hat{h}_{cd})\frac{\nabla_b\nabla_a \tau}{\sqrt{1+|\nabla\tau|^2}}+ div_\Sigma (  \frac{h \nabla\tau}{\sqrt{1+|\nabla\tau|^2}}) +  div_{\sigma} \alpha \\
={} & \nabla_b   [   (\widehat{H}\hat{\sigma}^{ab} -\hat{\sigma}^{ac} \hat{\sigma}^{bd} \hat{h}_{cd})   \sqrt{1+|\nabla\tau|^2} \hat \nabla_a  \tau +     \frac{h \tau^b}{\sqrt{1+|\nabla\tau|^2}}    + \alpha^b]  \\
={} &  \nabla_b   [  -\hat{\sigma}^{ac} \hat{\sigma}^{bd} \hat{h}_{cd}   \sqrt{1+|\nabla\tau|^2} \hat \nabla_a  \tau  - \frac{\tau^b}{1+|\nabla \tau|^2}  \alpha (\nabla \tau)  + \alpha^b] \\
={} &  \nabla_b   [  \frac{-  \hat{\sigma}^{bd} \hat{h}_{cd}    \tau^c}{  \sqrt{1+|\nabla\tau|^2}}  - \frac{\tau^b}{1+|\nabla \tau|^2} \alpha (\nabla \tau)  + \alpha^b]. \\
\end{split}  \]
On the other hand, by definition,  
\[  \alpha_a
= \frac{1}{\sqrt{1+|\nabla \tau|^2}}\langle  \nabla_{\frac{\partial}{ \partial x^a}} \hat \nu,  T_0 + \nabla \tau \rangle
=  \frac{\hat h_{ac} \tau^c}{\sqrt{1+|\nabla \tau|^2}}.
  \]
Hence,
\[  - \frac{\tau^b}{1+|\nabla \tau|^2}  \alpha (\nabla \tau)  + \alpha^b = \frac{\sigma^{ab}\hat h_{ac} \tau^c}{\sqrt{1+|\nabla \tau|^2}} -\frac{\tau^b}{1+|\nabla \tau|^2}\frac{\hat h_{ac} \tau^a \tau^c}{\sqrt{1+|\nabla \tau|^2}} = \frac{ \hat \sigma^{ab}\hat h_{ac} \tau^c}{\sqrt{1+|\nabla \tau|^2}}. \]
This shows that 
\[  \frac{-  \hat{\sigma}^{bd} \hat{h}_{cd}    \tau^c}{  \sqrt{1+|\nabla\tau|^2}}  - \frac{\tau^b}{1+|\nabla \tau|^2} \alpha (\nabla \tau)  + \alpha^b =0 \]
and completes the proof of the lemma.
\end{proof}

\begin{lemma} \label{thm_positive}

Let $X_{s\tau}, 0\leq s \leq 1$ be a family of isometric embeddings of $\sigma$ into $\R^{3,1}$ with time function $s\tau$. Suppose $\Sigma_0$, 
the image of $X_0$, lies in a totally geodesic Euclidean 3-space, $E^3$, and  $\Sigma_{s\tau}$, the image of $X_{s\tau}$, projects to an embedded surface in $E^3$ for $0\leq s \leq 1$.
Assume further that $\Sigma_0$ is mean convex  and  $H_0^2 \geq \langle H_{s\tau},   H_{s\tau}  \rangle$  for $0\leq s\leq 1$. 
Regarding $\Sigma_0$ as a physical surface in the spacetime
$\R^{3,1}$, then
\[  E(\Sigma_0,\tau)\ge 0. \]
\end{lemma}
\begin{proof}
Instead of proving admissibility of $\tau$, we will use variation of the quasi-local energy and the point-wise inequality of the mean curvatures.

 Hence, we consider \[ F(s)=E(\Sigma_0,s\tau) \] for $0 \le s \le 1$. $F(0)=0$ and we shall prove that $F(1)$ is non-negative by deriving a differential inequality for $F(s)$. 
 
 Let 
\[ G(s) = \int_{\widehat \Sigma_{s\tau}}  \widehat H_{s\tau} dv_{\widehat \Sigma_{s\tau}}.\] For any $0\leq s_0\leq 1$,  $G(s)$ is related to $\tilde{E}(\Sigma_{s_0\tau}, \breve{e}_3(\Sigma_{s_0\tau}), s\tau)$ by
\begin{equation} \label{eqn_G}
G(s) =\tilde{E}(\Sigma_{s_0\tau}, \breve{e}_3(\Sigma_{s_0\tau}), s\tau)+ \int_{ \Sigma_{s_0\tau}}   \left(  -\langle H_{s_0\tau} ,  \breve{e}_3(\Sigma_{s_0\tau})\rangle  \sqrt{1+|\nabla s \tau|^2}  - \alpha_{\breve{e}_3(\Sigma_{s_0\tau})}   ( \nabla s\tau) \right)   dv_{\Sigma_{s_0\tau}}.
\end{equation}

As a consequence of Lemma \ref{lemma_cri}, we have
\[   \frac{\partial}{ \partial s} \tilde E(\Sigma_{s_0\tau},  \breve{e}_3(\Sigma_{s_0\tau}), s\tau) \Big{|} _{s=s_0}=0.\]
By equation \eqref{eqn_G}, 
\[ 
 \begin{split}
G'(s_0) 
= &   \int_{\Sigma_{s_0\tau}}  \left(\frac{-\langle H_{s_0 \tau} ,  \breve{e_3}(\Sigma_{s_0\tau})\rangle  s_0 |\nabla\tau|^2}{\sqrt{1+s_0^2|\nabla\tau|^2} } - \alpha_{\breve{e_3}(\Sigma_{s_0\tau})}   ( \nabla \tau) \right) \,\,dv_{\Sigma_{s_0\tau}} \\
=& \frac{G(s_0)}{s_0} -\frac{1}{s_0}  \int_{\Sigma_{s_0\tau}} \frac{1}{\sqrt{1+s_0^2|\nabla\tau|^2}} \sqrt{  \langle H_{s_0 \tau} , H_{s_0 \tau} \rangle+ \frac{(s_0\Delta \tau)^2}{1+|s_0\nabla \tau|^2}   }\,\, dv_{\Sigma_{s_0\tau}}.
 \end{split}\]
Recall that by the definition of quasi-local energy:
\[ F(s) = G(s) -  \int_{\Sigma_0}\left( \sqrt{(1+|s\nabla\tau|^2) H_0^2 + (s\Delta \tau)^2 } -s\Delta \tau \sinh^{-1}(\frac{s\Delta \tau}{ H_0 \sqrt{1+|s\nabla\tau|^2}})\right) dv_{\Sigma_0}. \] 

We write the integrand of the last integral as 
\[ \sqrt{1+s^2|\nabla\tau|^2}  \left(\sqrt{H_0^2+ \frac{(s\Delta \tau)^2}{1+|s\nabla \tau|^2}   }
       -  \frac{(s\Delta \tau)}{\sqrt{1+|s\nabla \tau|^2}}\sinh^{-1}(   \frac{s\Delta \tau}{H_0\sqrt{1+|s\nabla \tau|^2}}) \right).\] Differentiate this expression with respect to the variable $s$, we obtain 
       
\[  \begin{split}
&     \frac{s |\nabla\tau|^2}{\sqrt{1+s^2|\nabla\tau|^2} } \left (\sqrt{H_{0}^2+ \frac{(s\Delta \tau)^2}{1+|s\nabla \tau|^2}   }
 -  \frac{(s\Delta \tau)}{\sqrt{1+|s\nabla \tau|^2}}\sinh^{-1}(   \frac{s\Delta \tau}{H_{0}\sqrt{1+|s\nabla \tau|^2}}) \right) \\
   & - \sqrt{1+s^2|\nabla\tau|^2} \left[ \frac{s \Delta \tau}{ (1+ | s \nabla \tau|^2)^{3/2}}  \sinh^{-1} (   \frac{s\Delta \tau}{H_{0}\sqrt{1+|s\nabla \tau|^2}}  ) \right] \\
=&  \frac{1}{s}   \left[  \sqrt{(1+|s\nabla\tau|^2){H_0}^2 + (s\Delta \tau)^2 } -s\Delta \tau \sinh^{-1}(\frac{s\Delta \tau}{ {H_0} \sqrt{1+|s\nabla\tau|^2}})  \right]   \\
  &-\frac{1}{s}   \frac{1}{\sqrt{1+s^2|\nabla\tau|^2}} \sqrt{H_{0}^2+ \frac{(s\Delta \tau)^2}{1+|s\nabla \tau|^2} }.
 \end{split}\]

Since the induced metrics on $\Sigma_{s_0\tau}$ and $\Sigma_0$ are the same, we can evaluate all integrals on the surface $\Sigma_0$. This leads to
\[  F'(s) = \frac{F(s)}{s} + \frac{1}{s}      \int_{\Sigma_0} \left(\frac{ \sqrt{H_0^2+ \frac{(s\Delta \tau)^2}{1+|s\nabla \tau|^2}   } - \sqrt{ \langle H_{s_ \tau} , H_{s_ \tau} \rangle+ \frac{(s\Delta \tau)^2}{1+|s\nabla \tau|^2}   } }{\sqrt{1+s^2|\nabla\tau|^2}}\right)dv_{\Sigma_0} .  \]

The assumption $H_0^2  \geq \langle H_{s \tau} , H_{s \tau} \rangle $ implies the last term is non-negative and thus
\[ F'(s) \ge \frac{F(s)}{s}. \]
As $F(0) = F'(0)=0$, the positivity of $F(s)$ follows from a simple comparison result for ordinary differential equation. 
\end{proof}

The last lemma specializes to axially symmetric metrics.  

 \begin{lemma} \label{lemma_mean_identity}
Suppose the isometric embedding $X_0$ of an axially symmetric metric $\sigma=P^2 d \theta^2 + Q^2 \sin^2 \theta d \phi^2$ into $\R^3$ is given by the coordinates $(u \sin \phi, u \cos \phi , v)$  where $P$, $Q$, $u$, and $v$ are functions of $\theta$. Let $\tau=\tau(\theta)$ be an axially symmetric function and $X_\tau$ be the isometric embedding of $\sigma$ in $\R^{3,1}$ with time function $\tau$.
The following identity holds for the  mean curvature vector $H_{ \tau}$ of $\Sigma_{ \tau}$ in $\R^{3,1}$.
\[
    \langle H_\tau , H_\tau \rangle   =H_0^2  -   \frac{  ( v_\theta  \Delta \tau - \tau_\theta \Delta v )^2}{  v_\theta^2 + \tau_{\theta}^2}
\] where $\Delta$ is the Laplace operator of $\sigma$.
\end{lemma}
\begin{proof}
The isometric embedding for an axially symmetric metric is reduced to solving ordinary differential equations. The isometric embedding  of $\sigma$ into $\R^3$ is given by $( u \sin \phi  , u \cos \phi ,v)$
where
\[ u^2 =Q^2 \sin^2 \theta  \qquad {\rm and}  \qquad v_\theta^2 + u_\theta ^2 = P^2.\]
The isometric embedding of the metric $\sigma + d\tau \otimes d\tau$ into $\R^3$ is given by 
 $( u \sin \phi  , u \cos \phi , \tilde v)$
where
\[  \tilde v_\theta^2 + u_\theta ^2 = P^2+ \tau_{\theta}^2. \]
Thus,
\[    \tilde v_\theta = \sqrt{  v_\theta^2 + \tau_{\theta}^2 } .  \]
Differentiating one more time with respect to $\theta$,
\[    \tilde v_{\theta\theta} = \frac{ v_\theta  v_{\theta\theta}    + \tau_\theta \tau_{\theta\theta}    }{\sqrt{  v_\theta^2 + \tau_{\theta}^2 } }, \]
and therefore
\[   \Delta   \tilde v  = \frac{1}{\sqrt{  v_\theta^2 + \tau_{\theta}^2 }}( v_\theta \Delta v + \tau_\theta \Delta \tau  ). \]
For the mean curvature, we have
\[ 
\begin{split}  
H_0^2 = &( \Delta(u \sin \phi))^2 + ( \Delta(u \cos \phi))^2 + ( \Delta v)^2, \text{ and }\\ 
\langle H_\tau , H_\tau \rangle =& ( \Delta(u \sin \phi))^2 + ( \Delta(u \cos \phi))^2 + ( \Delta \tilde v)^2 - ( \Delta \tau)^2. 
\end{split}
\]
Taking the difference and completing square, we obtain
\[ \begin{split} 
      H_0^2  -  \langle H_\tau , H_\tau \rangle 
=  & ( \Delta v)^2 +  ( \Delta \tau)^2 -   ( \Delta \tilde v)^2  \\
=  & \frac{1}{  v_\theta^2 + \tau_{\theta}^2} [ v_\theta  \Delta \tau - \tau_\theta \Delta v ]^2.
\end{split}\]
\end{proof}

Let's recall the statement of Theorem \ref{thm_global_axial}.

\vspace{5mm}\noindent{\bf Theorem 3}
{\it Let  $\Sigma$ satisfy Assumption 1. Suppose that the induced metric $\sigma$ of  $\Sigma$  is axially symmetric with positive Gauss curvature, $\tau=0$ is a solution to the optimal embedding equation for $\Sigma$ in $N$, and
\[ H_0 > |H|>0. \]
Then for any axially symmetric time function $\tau$ such that  $\sigma + d\tau \otimes  d\tau$ has positive Gauss curvature, 
\[ E(\Sigma , \tau) \ge E(\Sigma, 0). \]
Moreover, equality holds if and only if $\tau$ is a constant . } 

\begin{proof} By Theorem \ref{thm_int}, it suffices to show that $E(\Sigma_0, \tau)\geq 0$.
First, we show that for any $0 \le s \le 1$, the isometric embedding with time function $s \tau$ exists. Recall the Gaussian curvature for the metric $ \sigma + d \tau \otimes d \tau$ is
\[  \frac{1}{1+| \nabla \tau|^2} \left[ K+  ( 1+| \nabla \tau|^2 )^{-1}det(\nabla^2 \tau) \right] \]
where $K$ is the Gaussian curvature for the metric $ \sigma$. Since $K$ and $K+  ( 1+| \nabla \tau|^2 )^{-1}det(\nabla^2 \tau)$ are both positive, we conclude that $ \sigma + d (s \tau) \otimes d (s\tau)$   has positive Gaussian curvature for all   $0 \le s \le 1$. By Lemma \ref{lemma_mean_identity}, we have $H^2_0\geq  \langle H_{s\tau},  H_{s\tau} \rangle$. The theorem now follows from Lemma \ref{thm_positive}.
\end{proof}

\appendix 
\section{Proof of Lemma \ref{lemma_gauge} }
{\rm Here we present the proof of Lemma \ref{lemma_gauge} used in the proof of Theorem 2. Since the lemma is a general statement for any time function $\tau_0$, we use $\tau$ instead of $\tau_0$. 
\begin{lemma} \label{lemma_gauge}
For a surface $\Sigma_{\tau}$ in $\R^{3,1}$ which bounds a spacelike hypersurface $\Omega$, let $f$ be the solution of Jang's equation on $\Omega$ with boundary value $\tau$. 
Then 
\[ h(\Sigma, X_{\tau}, \tau , e'_3)   =h(\Sigma, X_{\tau}, \tau , \breve e_3(\Sigma_{\tau}))  \]
\end{lemma}

\begin{proof}
It suffices to show that $e_3'= \breve e_3(\Sigma_{\tau})$. For simplicity, denote $ \breve e_3(\Sigma_{\tau})$ by $ \breve e_3$ in this proof.

Let $\widehat \Omega$ and $\widehat \Sigma$ denote the projection of $\Omega$ and its boundary, $\Sigma_{\tau}$, to the complement of $T_0$. 
Let $\hat \nabla$ be the covriant derivative on $\widehat \Sigma$ and $\hat D$ be the covariant derivative on $\widehat \Omega$. Let $\nabla$ be the covriant derivative on $\Sigma$.

Write $ \Omega$  as the graph over $\widehat \Omega$ of the function $f$. $f$ can be viewed as a function on $\Omega$ as well. $f$ is precisely the solution to Jang's equation on $\Omega$ with Dirichlet boundary data $\tau$. 

We choose an orthonormal frame $\{ \hat e_a \}$ for $T \widehat \Sigma $. Let $\hat e_3$ be the outward normal of $\widehat \Sigma$ in $\widehat \Omega$. $\{  \hat e_a, \hat e_3, T_0 \}$ forms an orthonormal frame of the tangent space of $\R^{3,1}$. The frame is extend along $T_0$ direction by parallel translation to a frame of the tangent space of $\R^{3,1}$ on $\Sigma$. 

Let $\{ e_3, e_4 \}$ denote the frame of the normal bundle of $\Sigma$ such that $e_3$ is the unit outward normal of $\Sigma$ in $\Omega$ and $e_4$ is the furture directed unit normal of $\Omega$ in $\R^{3,1}$. In terms of the frame $\{  \hat e_a, \hat e_3, T_0 \}$, 
\begin{equation} \label{equ_frame_1}
\begin{split}
 e_3= & \frac{1}{\sqrt{1-|\hat D f|^2}} \left [    \sqrt{1- |\hat \nabla \tau|^2} \hat e_3 + \frac{\hat e_3(f)}{\sqrt{1- |\hat \nabla \tau|^2}}(T_0 + \hat  \nabla \tau) \right ] \\
 e_4 = & \frac{1}{ \sqrt{1- |\hat D f|^2}}(T_0 + \hat D f) 
\end{split}
\end{equation}
Let $\{ e_3', e_4' \}$ denote the frame determined by Jang's equation.  By Definition 5.1 of \cite{wy2}, it is chosen such that  
\[  \langle e_3 , e'_4 \rangle = \frac{-e_3(f)}{\sqrt{1+ |\nabla \tau|^2}} \]
Using equation (\ref{equ_frame_1}), 
\[  \langle e_3 , e'_4 \rangle = \frac{-e_3(f)}{\sqrt{1+ |\nabla \tau|^2}}=  \frac{-\hat e_3(f)}{\sqrt{1+ |\nabla \tau|^2}\sqrt{1-|\hat D f|^2}\sqrt{1-|\hat \nabla \tau|^2}}  \]
As a result, 
\[  \langle e_3 , e'_4 \rangle =  \frac{-\hat e_3(f)}{\sqrt{1-|\hat D f|^2}}  \]
since $ ( 1- |\hat \nabla \tau|^2 )(1+ |\nabla \tau|^2)=1.$

On the other hand, the frame $\{ \breve e_3, \breve e_4  \}$ is the frame of normal bundle such that $ \breve e_3 = \hat e_3$. In terms of the frame $\{  \hat e_a, \hat e_3, T_0 \}$,
\[  \breve e_4= \frac{T_0 + \hat \nabla \tau}{\sqrt{1-|\hat \nabla \tau|^2}}  \]
Using equation (\ref{equ_frame_1}),  
\[  \langle  e_3 ,  \breve e_4 \rangle =  \frac{1}{\sqrt{1-|\hat \nabla \tau|^2}}\left [\frac{\hat e_3(f) (-1+|\hat \nabla \tau|^2) }{\sqrt{1-|\hat D f|^2}\sqrt{1-|\hat \nabla \tau|^2} }\right ] = \frac{-\hat e_3(f)}{\sqrt{1-|\hat D f|^2}} \]
As a result, 
\[   \langle e_3 ,  \breve e_4 \rangle =   \langle e_3 , e'_4 \rangle \] 
Hence, 
$\{ e_3', e_4' \}$ and  $\{ \breve e_3, \breve e_4  \}$  are the same frame for the normal bundle of $\Sigma$
\end{proof}

{
\end{document}